\documentclass[11pt]{amsart}
\usepackage{geometry} 
\usepackage{graphicx}
\usepackage[colorlinks]{hyperref}
\usepackage[hyperpageref]{backref}

\usepackage{enumitem}

\usepackage[draft]{fixme}

\newtheorem{thm}{Theorem}[section]
\newtheorem{theorem}[thm]{Theorem}
\newtheorem{lemma}[thm]{Lemma}
\newtheorem{proposition}[thm]{Proposition}
\newtheorem{corollary}[thm]{Corollary}

\newtheorem{example}[thm]{Example}
\newtheorem{remark}[thm]{Remark}
\newtheorem{defn}[thm]{Definition}

\numberwithin{equation}{section}

\newcommand{\R}{{\mathbb{R}}}

\title{A reverse isoperimetric inequality for planar $(\alpha,\beta)-$convex bodies}

\author[Croce]{Gisella Croce}

\author[Fattah]{Zakaria Fattah}

\author[Pisante]{Giovanni Pisante}

\address[G. Croce]{Normandie Univ, UNIHAVRE, LMAH, FR-CNRS-3335, 76600 Le Havre, France}
\email{gisella.croce@univ-lehavre.fr}

\address[Z. Fattah]{Mathematics and Computer Science Department, ENSAM of MeknÃšs, University of Moulay Ismail,
Marjane II, AL Mansour, B.P 15290, 50050 Meknes, Morocco}
\email{z.fattah@edu.umi.ac.ma}

\address[G. Pisante]{Dipartimento di Matematica e Fisica, Universit\`a degli Studi dell Campania "Luigi Vanvitelli", Viale Lincoln 5, 81100 Caserta, Italy}
\email{giovanni.pisante@unicampania.it}

\begin{document}

\maketitle

\begin{abstract}
In this paper, we study a reverse isoperimetric inequality for planar convex bodies whose radius of curvature is between two positive numbers $0\leq \alpha <\beta$, called $(\alpha,\beta)-$convex bodies. We show that among planar $(\alpha,\beta)-$convex bodies of fixed perimeter, the extremal shape is a domain whose boundary is composed by two arcs of circles of radius $\alpha$ joined by two arcs of circles of radius $\beta$.
\end{abstract}

\section{{Introduction}}
The planar version of the isoperimetric inequality  states that  the circle encloses the maximal area among all curves of the same length. 
In this paper we deal with the reverse problem of finding, among the convex curves of given length and satisfying a constraint on the curvature, the one enclosing the minimal area. One of the first results in this framework has been obtained by A. Borisenko and K. Drach in \cite{Borisenko_2014}. They proved that for any the class of convex set $K$ whose curvature $k$ satisfies $k\geq \lambda >0$ the  inequality
\begin{equation}
    \label{BorDr}
    A(K)\geq \frac{P(K)}{2\lambda}-\frac{1}{\lambda^2}\sin\left(\frac{P(K)\lambda}{2}\right)
\end{equation}
holds, where $A(K)$ and $P(K)$ denote the area and the perimeter of the set $K$ respectively. Moreover they studied the equality case proving that the equality is satisfied if and only if $K$ is the intersection of two disks of radius $\lambda$, called $\lambda$-lune. On the other side R. Chernov, K. Drach and K. Tatarko in \cite{CHERNOV2019431} proved the reverse isoperimetric inequality for convex sets with the curvature satisfying $k \le \lambda < \infty$ showing that 
\begin{equation}
    \label{CerTat}
A(K)\geq  \frac{P(K)}{\lambda}-\frac{\pi}{\lambda^2}.
\end{equation}
Also in this case the sets satisfying the equality in the previous inequality has been characterized and shown to be the convex hull of two balls of radius $\frac{1}{\lambda}$.

In this paper we generalize the inequality \eqref{BorDr} to the class of convex sets with curvature satisfying $\alpha \leq k \leq \beta$ for given constants $0 \leq \alpha < \beta$, by proving that 
\begin{equation}
    \label{Isop-new}
A (K)\geq \frac 12 (\beta+\alpha)(P(K)-2\pi \alpha)+\pi \alpha^2- (\beta-\alpha)^2 \sin\left(\frac{P(K)-2\pi \alpha}{ 2(\beta-\alpha)}\right).
\end{equation}
Moreover the equality holds if and only if the boundary of $K$ is composed by two arcs of circles of radius $\alpha$ joined by two arcs of circles of radius $\beta$. We explicitly observe that formally for $\beta \to +\infty$ one gets the inequality \eqref{CerTat}.
 
The idea of the proof stems from the argument used in \cite{Borisenko_2014} based on the Pontryagin Maximum Principle to derive the optimality conditions satisfied by a constrained minimizer of the area. Then we exploit these  conditions in an analytical way to establish the result, as opposed to the more geometrical analysis done in \cite{Borisenko_2014}.

It is worth to note that imposing some constraints on the class of all admissible curves, is necessary, since otherwise the problem can be easily proved to be ill-posed.
In the literature different reverse isoperimetric problems have been considered (see for instance  \cite{MR1136445, Zhang_2019, MR3078385, blatt2020reverse, paoli2021reverse, MR4234702, Howard_1995,Kostia} for different approaches and different classes of admissible sets).

\section{Preliminaries}
\label{SEC:prelim}

With the notation $\mathfrak{S}(A,B)$, where $A$ and $B$ are two sets and $\mathfrak{S}$ is one of the standard symbol for a functional space (such as $C^{0,1}$ for Lipschitz functions, $W^{1,1}$ for absolutely continuous, etc.), we mean the space of maps defined on $A$ with values in $B$. Moreover for $I=[a,b]\subset \mathbb{R}$ with $\mathfrak{S}_{per}(I,B)$ we denote the $(b-a)$-periodic functions in $\mathfrak{S}_{per}(\mathbb{R},B)$.
 
\subsection{Convex bodies and $(\alpha,\beta)$-convexity}

In this section we recall some basic properties of convex sets in euclidean spaces, we introduce the class of competitors for our optimisation problem and prove some useful properties such as regularity and  compactness. 

Throughout the paper we will denote with $B_{r}$ the closed ball with radius $r>0$ centred in the origin.
By a convex body we shall mean a compact convex set $K\subset \R^{n}$ with non-empty interior. 
With $\mathcal{K}^{n}$ we will denote the class of convex bodies in $\mathbb{R}^{n}$. 
The {\it support function} of $K$ is the real-valued function defined on the unit sphere $\mathbb{S}^{n-1} $ by
\[
h_{K}(v):=\max_{k \in K} \langle k, v \rangle ,\quad v \in \mathbb{S}^{n-1}.
\]
We recall that the support function $h_{K}$ characterises the set $K$ and any function $h\;:\; \mathbb{S}^{n-1} \to \mathbb{R}$, such that its $1$-homogeneous extension is convex, is the support function of a convex body (cf. \cite[\S 1.7]{MR3155183}). Moreover $K\in \mathbb{R}^{n}$ is strictly convex if and only if  the $1$-homogeneous extension of its support function belongs to $C^{1}(\mathbb{R}^{n}\setminus \{0\})$ (see \cite[Cor. 1.7.3 \& \S 2.5]{MR3155183}).

A convenient way to endow $\mathcal{K}^{n}$ with a topology is to use the Hausdorff distance between two non-empty compact sets, denoted by $d^{H}(\cdot, \cdot)$ (cf. \cite[\S 1.8]{MR3155183}).  Indeed, we recall that the perimeter and the area functionals are continuous with respect to the Hausdorff topology on $\mathcal{K}^{n}$ (see also \cite[Thm.23 and Thm. 26]{MR2428231}).  By \cite[Lem. 1.8.14]{MR3155183}, given $K,M\in \mathcal{K}^{n}$, we can characterize the Hausdorff distance of $K$ from $M$ in terms of their support functions:
\begin{equation}\label{dist_Hausdorff_supp_functions}
d^{H}(K,M)=\left \| h_K-h_M \right\|_{L^{\infty}(\mathbb{S}^{n-1})}.
\end{equation}
Moreover, by the {\it Blaschke selection theorem} (cf. \cite[Theorem 1.8.7]{MR3155183}), every bounded sequence of convex bodies has a subsequence that converges to a convex body in the Hausdorff topology.

Following \cite[pag. 157]{MR3155183} we say that the convex body $L$ is {\it locally embeddable} in the convex body $K$ if for each point $x\in \partial K$ there are a point $y\in L$ and a neighbourhood $U$ of $y$ such that
\[
(L\cap U) +x-y \subset K.
\]

The concept of local embeddability when $K$ or $L$ is a ball has been studied and used in several contexts (cf. Remark \ref{rem-equivdef} below). The following lemmata provide two classical regularity properties related to this concept.

\begin{lemma}\label{Regularity-A}
If a convex body $K$ is locally embeddable in a ball, then its support function is of class $C^{1,1}$.
\end{lemma}
\begin{proof}
From the local embeddability of $K$ in a ball, say $B$, it follows that $K$ is a strictly convex body. Therefore by \cite[Theorem 3.2.3]{MR3155183}, there exists a convex body $M\in \mathcal{K}^{n}$ such that $B=K+M$ (i.e. $K$ is a {\it summand} of $B$) which is equivalent to say that $K$ {\it slides freely inside} $B$ (cf. \cite[Theorem 3.2.2]{MR3155183}). The result follows by the characterization of convex bodies with support function of class $C^{1,1}$ (cf. \cite[Proposition 2.3]{Howard_2006}).
\end{proof}
\begin{lemma}\label{Regularity-B}
Let $K$ be a convex body. If a ball is locally embeddable in $K$, then its boundary $\partial K$ is of class $C^{1,1}$.
\end{lemma}
\begin{proof}
Let $B$ a ball locally embeddable in $K$. Since $B$ is strictly convex, by \cite[Theorem 3.2.3]{MR3155183} there exists a convex body $M\in \mathcal{K}^{n}$ such that $K=B+M$ (i.e. $B$ is a {\it summand} of $K$), that is equivalent to $\partial K$ being of class $C^{1,1}$ (cf. for example \cite[Proposition 2.4.3]{MR2311920}).
\end{proof}
We are now ready to give the notion of curvature that we are going to use for our result:
\begin{defn}
Let $\alpha$ and $\beta$ be two real numbers with $0<\alpha<\beta$. We say that a convex body $K$ is {\it $(\alpha,\beta)-$convex} if $K$ is locally embeddable in $B_{\beta}$ and $B_{\alpha}$ is locally embeddable in $K$.
In the case $\alpha=0$, we will assume that $K$ is locally embeddable in $B_{\beta}$.
\end{defn}
\begin{remark}
Note that if $\alpha =0$ the class of $(\alpha,\beta)-$convex sets is nothing but the family of $\beta-$convex sets introduced in \cite{Borisenko_2014}. From the previous lemmata it follows that for $\alpha>0$,  an $(\alpha,\beta)-$convex body is a $C^{1,1}$ strictly convex set with support function of class $C^{1,1}$. In the case $\alpha=0$, we can only say that the support function of its boundary is $C^{1,1}$.
\end{remark}

\begin{remark}\label{rem-equivdef}
Let $\alpha>0$.
An equivalent definition of $(\alpha,\beta)-$convexity can be given in terms of the following notions of $\beta-$convexity and $\alpha-$concavity. A convex body $K$ is said to be {\it $\alpha-$concave} if the ball of radius $1/\alpha$ is locally embeddable in $K$ (cf. \cite[Definition 1.2]{CHERNOV2019431}). While $K$ is said to be {\it $\beta-$convex} if for each point $y \in \partial K $, there exist a point $y\in \partial B_{\beta}$ and a neighbourhood $U$ of $x$ such that $(K\cap U)+y-x\subset B_{\beta}$. As a matter of fact, as a consequence of \cite[Theorem 1.9]{MR880294}, a convex body $K$ is $\beta-$convex if and only if $K$ is locally embeddable in $B_{\beta}$. Therefore $K$ is $(\alpha,\beta)-$convex if it is at the same time $\beta-$convex and $1/\alpha-$concave.
\end{remark}

We will mostly work in the two dimensional setting, dealing with planar convex sets, therefore we recall some useful preliminaries results on planar convex geometry. First we note that it is often convenient to work with the so called {\it parametric support function}, i.e. $p_{K}(t):=h_{K} \circ \boldsymbol\sigma(t)$ where $\boldsymbol\sigma(t)=(\cos(t),\sin(t))$ and $t\in [0, 2\pi]$ (note that this is how the support function of a planar convex body is defined in classical literature, cf. for example \cite{MR2162874} and \cite{MR0170264}). In the next propositions we recall some well known and useful properties related to the parametric support function of a planar convex body (see also \cite{MR3113290} for a recent survey on the subject). We remark that similar results holds under weaker assumption regularity assumptions (cf. \cite{MR1242973}, \cite{MR3155183}) but we restrict our attention to what will be sufficient for our purposes.

\begin{proposition}
\label{Properties-01}
Let $K\in \mathcal{K}^{2}$ be a strictly convex planar body  and $p_{K}$ its parametrized support function. Assume that  $p_{K} \in C^{1,1}_{per}(0,2\pi)$. Then the radius of curvature of the boundary $\partial K$, $\rho_{K}(t)$, satisfies for a.e. $t\in (0, 2\pi)$ the equation
\begin{equation}
\label{EQ-curvature}
\rho_{K}(t)= p_{K}(t)+ p_{K}''(t) \geq 0.
\end{equation}
Viceversa, if $h \in  C^{1,1}_{per}(0,2\pi)$ is a function satisfying \eqref{EQ-curvature}, then there exists a convex body $K$ such that $h$ is its parametric support function.
\end{proposition}

\begin{proposition}
\label{Properties-02}
Under the same assumptions of Proposition \ref{Properties-01} the boundary $\partial K$ can be parametrized by 
\[
\begin{cases}
x(t) = p_{K}(t) \cos(t) - p_{K}'(t) \sin(t)  \\
y(t) = p_{K}(t) \sin(t) + p_{K}'(t) \cos(t)
\end{cases}
\]
Moreover the perimeter and the area of $K$ can be computed by the following formulae 
\[
P(K)= \int_{0}^{2\pi} \left( p_{K}(t)+ p_{K}''(t) \right) dt =  \int_{0}^{2\pi} \rho_{K}(t) dt ,
 \]  
\[
A(K)= \frac{1}{2}\int_{0}^{2\pi} \left( p_{K}(t)+ p_{K}''(t) \right)  p_{K}(t)\, dt = \frac{1}{2}\int_{0}^{2\pi} \rho_{K}(t) \, p_{K}(t)\, dt.
\]  
\end{proposition}

The $(\alpha,\beta)$-convexity for planar domains can be expressed in terms of parametrized support function. Indeed a convex body $K\in \mathcal{K}^2$ is $(\alpha,\beta)$-convex if and only if 
its parametrized support function $p_K$ satisfies the inequalities
\[
\alpha \leq p_{K}(t)+ p_{K}''(t) \leq \beta \quad \textnormal{a.e. in}\, (0,2\pi). 
\]
It easily follows, by Proposition \ref{Properties-02}, that the perimeter of any $(\alpha,\beta)$-convex body in the plane  satisfies $ 2\pi \alpha \leq P(K)\leq 2\pi \beta$. 
Moreover this characterization allows us to prove that for planar domains the $(\alpha,\beta)$-convexity is preserved by Hausdorff convergence.
 \begin{lemma}\label{lemma-continuity}
 Let $0\leq \alpha<\beta<\infty$ and $\{K_n\}_{n \in \mathbb{N}} \subset \mathcal{K}^2$ be a sequence of $(\alpha,\beta)$-convex bodies. Assume that the sequence $K_n$ converges to $K$ in the Hausdorff
topology. Then $K$ is $(\alpha,\beta)$-convex.
 \end{lemma}
\begin{proof}
 Let ${p_{K_n}}$ be the parametric support function of $ K_{n}$ and $p_K$ the parametric support function of $K$. As $ K_{n}$ is $(\alpha,\beta)-$convex body then $p_{K_n}$ is of class $C^{1,1}$ and the radius of curvature of $\partial K_{n}$  exists almost everywhere. Therefore, the parametric support function satisfies
  \begin{equation}\label{LemmaMC}
   \alpha \le p_{K_n} +p''_{K_n} \le \beta  \quad \quad\mbox{a.e. in }  \quad
[0,2\pi].
  \end{equation}
  By formula \ref{dist_Hausdorff_supp_functions},  as $K_n$ converges to $K$ in the Hausdorff distance, $p_{K_n}$ converges to $p_K$ in $L^{\infty}([0,2\pi])$.
Since $p_{K_n}$ is bounded in $L^{\infty}([0,2\pi])$, we deduce from inequality (\ref{LemmaMC})
that $p''_{K_n}$ is bounded in $L^{\infty}([0,2\pi])$.

 By Proposition \ref{Properties-02}  
 $$
 p'_{K_n}(t)=-x_n(t)\sin(t)+y_{n}(t)\cos(t)\,,
 $$
 where $(x_n(t), y_n(t))$ is the 
parametrization of the boundary of $K_n$. 
As $(x_n(t), y_n(t)) \in \partial K_n$ and  all the $K_n$ are contained in a ball, then $p'_{K_n}$ 
is bounded in $L^{\infty}([0,2\pi])$. Therefore $p'_{K_n}$ is bounded in the Sobolev space $W^{1,\infty}([0,2\pi])$. By  the Rellich-Kondrachov theorem, there exists $w\in W^{1,\infty}([0,2\pi])$ such that, up to a subsequence,
$p'_{K} \overset{\ast}{\rightharpoonup} w$ in $W^{1,\infty}([0,2\pi])$ and $p'_{K_n} \rightarrow w$ in $L^{\infty}([0,2\pi])$.

Since $p_{K_n}$ is bounded in $W^{1,\infty}([0,2\pi])$, by using again the Rellich-Kondrachov  theorem, there exists $g\in W^{1,\infty}([0,2\pi])$ such that, up to a subsequence, $p_{K_n} \overset{\ast}{\rightharpoonup} g$ in $W^{1,\infty}([0,2\pi])$
and $p_{K_n} \rightarrow g$ in $L^{\infty}([0,2\pi])$.

Since $K_n\to K$ in the Hausdorff distance, the limit of  $p_{K_n}$ to $p_{K}$ in $L^{\infty}([0,2\pi])$ implies that  $g=p_K$, $w=g'$. Thus, we can extract a subsequence still denoted  $p'_{K_n}$ such that
 $p'_{K_n} \overset{\ast}{\rightharpoonup} p'_{K}$ in $W^{1,\infty}([0,2\pi])$.
 This implies that $p_{K} \in W^{2,\infty}([0,2\pi])$ which means $p_{K}$ is $C^{1,1}$.
As $p_{K}$ is the parametric support function of $K$, then $p_{K}$ and $p''_{K}$ are
$2\pi-$ periodic. Moreover, (\ref{LemmaMC}) implies
  $$
  \int \alpha \phi \le \int(p_{K_n}+p''_{K_n})\phi\le \int \beta \phi \quad
\mbox{for
all} \,\,  \phi \mbox{ smooth and  non negative}.
  $$
  Using the weak-${*}$ convergence in $W^{2,\infty}([0,2\pi])$, we have
$$
 \int \alpha \phi \le \int( p_{K}+p''_{K})\phi\le \int \beta \phi \quad \quad\mbox{for
all }  \,\phi  \mbox{ smooth and  non negative} .
  $$
  Thus, by the  Fundamental lemma in the  calculus of variations, we deduce that
$$   
\alpha \le p_{K} +p''_{K} \le \beta  \quad \quad\mbox{a.e. in }  [0,2\pi],
$$
that is, $K$ is $(\alpha,\beta)-$convex.
\end{proof}

We conclude this section with an example of a family of planar $(\alpha,\beta)$-convex bodies that will play an important role in the sequel.

\begin{example}[$(\alpha,\beta)$-eggs] 
\label{exa-eggs}
For $\alpha\geq 0$, an example of $(\alpha,\beta)$-convex bodies in $\mathcal{K}^{2}$ 
is the family of sets that we will call 
$(\alpha,\beta)-$eggs. 
They are symmetric with respect to 
the Cartesian axes,
their boundary is composed
by 4 arcs of circles with radii $\alpha$ and $\beta$ alternatively and their 
centers are chosen in such a way to ensure the regularity of $\partial K$ (see Lemma \ref{Regularity-B}).
Given $0 < \alpha < \beta < +\infty$ and $l \in (\alpha \pi, \beta \pi)$, the $(\alpha,\beta)$-egg, with perimeter $P=2l$,  can be parametrised as follows. We set $\tau= \frac{1}{2}\frac{l-\pi \alpha}{\beta -\alpha}$, $\kappa_{1}= (\beta-\alpha) \cos(\tau)$ and $\kappa_{2}= (\beta-\alpha) \sin(\tau)$. 
We note that $\tau \in \left(0, \frac{\pi}{2}\right)$ and $\kappa_{1}\cdot \kappa_{2} > 0$. We define the points $\mathbf{c}_{1}=(-\kappa_{1},0)$, $\mathbf{c}_{2}=(0,\kappa_{2})$,  $\mathbf{c}_{3}=(\kappa_{1},0)$, $\mathbf{c}_{4}=(0,-\kappa_{2})$. Then the boundary of the $(\alpha,\beta)$-egg is parametrised by 
\[
\boldsymbol\gamma(t)= \
\begin{cases}
\mathbf{c}_{1} + \beta \boldsymbol \sigma(t),  & t \in (- \tau, \tau) \\
\mathbf{c}_{2} + \alpha \boldsymbol \sigma(t),  & t \in (\tau, \pi - \tau) \\
\mathbf{c}_{3} + \beta \boldsymbol \sigma(t),  & t \in (\pi - \tau,\pi + \tau) \\
\mathbf{c}_{4} + \alpha \boldsymbol \sigma(t),  & t \in (\pi + \tau,\pi- \tau)
\end{cases}.
\]
Observe that in the case $\alpha=0$, the $(\alpha,\beta)$-egg set reduces to the $\beta$-lune defined in \cite{Borisenko_2014}. Indeed the arcs of radius $0$ corresponds to the corner points.
\end{example}

\begin{remark}\label{REM-consecArc}
An $(\alpha,\beta)$-egg is an example of convex set whose radius of curvature $\rho_{K}$ is piecewise constant and assumes alternatively the two values $\alpha$ and $\beta$. One could consider in general a wider class of planar $(\alpha,\beta)$-convex sets that satisfy this property. i.e. considering the class of sets whose boundary is a finite union of arcs of circles with radii $\alpha$ and $\beta$. When $\alpha >0$, due to the regularity of the boundary given by Lemma \ref{Regularity-B}, one can easily infer that two consecutive arcs cannot have the same radius of curvature and at least four arcs are needed. It follows therefore that the arcs forming the boundary of $K$ have to be even in number. In the case $\alpha=0$, since the support function is continuous, it follows the boundary consists of the arcs of the circle of radius $\beta$ 
adjoining each other at the corner points.
\end{remark}

\begin{example}[$(\alpha,\beta)$-regular $N$-gone] 
\label{N-gone} 
Given $0 \leq \alpha < \beta < \infty$ and $N\in \mathbb{N}$ with $N\geq 3$, we call { \it $(\alpha,\beta)$-regular $N$-gone} the $(\alpha,\beta)$-convex planar set $K$ whose boundary $\partial K$ is made up of $2N$ arcs of circles alternating the radii between $\alpha$ and $\beta$ and such that the length of all the arcs with the same radius is constant. In order to write the parametrized radius of curvature of a general $(\alpha,\beta)$-regular $N$-gone, $K$, fix $\sigma,\tau>0$ such that $N(\sigma+\tau)=2\pi$ and define, for $i\in \{1,2,\dots,2N\}$,
\[
t_i:=
\begin{cases}
\frac{i-1}{2} (\sigma+\tau) + \sigma & \text{if $i$ is odd} \\
\frac{i}{2} (\sigma+\tau) & \text{if $i$ is even} 
\end{cases}.
\]
The parametrized radius of curvature of $K$ can be written as
\[
\rho_{K}(t)=
\begin{cases}
\beta, &  t\in [t_{2i+1}, t_{2i+2}] \\
\alpha,&  t\in [t_{2i+2}, t_{2i+3}]
\end{cases}.
\]
Let $P(K)=L$ be the perimeter of $K$. By Proposition \ref{Properties-02} we easily get 
\begin{equation}
\label{EQ:per-N-gone}
P(K)=(\beta \sigma+ \alpha \tau)N=L
\end{equation}
and therefore
\begin{equation}
\label{EQ:sigma-tau}
\sigma= \frac{L-\alpha2 \pi}{N(\beta-\alpha)} \; ,\;\;\; \tau=\frac{2\pi \beta-L}{N(\beta-\alpha)}.
\end{equation}
The parametric support function of $K$ can consequently be written, by \eqref{EQ-curvature}, as 
\[
p_{K}(t)=
 \begin{cases}
   C_1^{2i+1} \cos t + C_2^{2i+1} \sin t + \beta \,,& t\in [t_{2i+1}, t_{2i+2}] \\ 
   C_1^{2i+2} \cos t + C_2^{2i+2} \sin t + \alpha \,,& t\in [t_{2i+2}, t_{2i+3}]
 \end{cases}, \;\; i\in\{0,1,\dots,N-1\}.
\]
Define $\lambda_{j}:= p_{K}(t_{j})$ for $j\in\{0,1,\dots,N-1\}$. 
The continuity of $p_{k}$ in any $t_{j}$ ensures us that
\[
C_1^{2i+1}= \frac{(\lambda_{2i+2}-\beta)\sin (t_{2i+1})-(\lambda_{2i+1}-\beta)\sin (t_{2i+2})}{\sin (t_{2i+2}  - t_{2i+1}) },
\]
\[
C_2^{2i+1}= \frac{(\lambda_{2i+2}-\beta)\cos (t_{2i+1})-(\lambda_{2i+1}-\beta)\cos (t_{2i+2})}{\sin (t_{2i+1}  - t_{2i+2} )},
\]
\[
C_1^{2i}= \frac{(\lambda_{2i}-\alpha)\sin (t_{2i+1})-(\lambda_{2i+1}-\alpha)\sin (t_{2i})}{\sin (t_{2i+1}  - t_{2i+2}) },
\]
\[
C_2^{2i}= \frac{(\lambda_{2i}-\alpha)\cos (t_{2i+1})-(\lambda_{2i+1}-\alpha)\cos (t_{2i})}{\sin (t_{2i}  - t_{2i+1} )}.
\] 
\end{example}
 
\subsection{Some easy consequences of the Pontryagin principle}

We will reformulate our constrained shape optimisation problem as an optimal control problem and we will exploit the optimality conditions given by the Pontryagin principle. The optimal control approach for shape optimisation problems is classical (see for example the monograph \cite{MR924574} for a wide introduction on the subject and \cite{MR2062547} for a more contemporary approach) and recently has been fruitfully applied to deal with constrained optimisation problems for convex domains (see  \cite{Bayen_2009}, \cite{bayen:tel-00212070}). Here we summarise the elementary notions on control theory and we state the version of Pontryagin optimality conditions suited for our purposes, considering indeed only autonomous problems with periodic phase variables valued in $\mathbb{R}^{2}$. 

Let $I=[a,b]\subset \mathbb{R}$ be a given interval, $J_{u}\subset \mathbb{R}$ be a compact set. For given maps $f, g \in C^{1}(\mathbb{R}^{3})$ and $\mathbf{h} \in C^{1}(\mathbb{R}^{3}, \mathbb{R}^{2})$ consider the problem of minimizing the functional 
\[
F(\mathbf{x},u) := \int_{a}^{b} f\big( \mathbf{x}(t),u(t) \big) dt  
\]
among all pairs $\big(\mathbf{x}(t),u(t)\big)\in W^{1,1}_{per}(I,\mathbb{R}^{2})\times L^{\infty}(I,J)$ that satisfy for almost every $t\in I$ the differential constraint
\begin{equation}\label{state1}
\mathbf{x}'(t)= \mathbf{h} \big( \mathbf{x}(t),u(t) \big)
\end{equation}
as well as the integral constraint
\begin{equation}\label{integral-constraint}
G(\mathbf{x},u) := \int_{a}^{b} g\big( \mathbf{x}(t),u(t) \big) dt  = C_{0}
\end{equation}
for a given constant $C_{0}$. The previous constrained extremal problem is a typical example of {\it optimal control problem}, $u$ in the so called {\it control variable}, $\mathbf{x}$ takes the name of {\it phase variable} and any pair $(\mathbf{x},u)$ that satisfy \eqref{state1} will be called a {\it controlled process}. A controlled process that minimises (locally in a $C(I)$-neighbourhood of $\mathbf{x}$) the functional $F(\mathbf{x},u)$ among the controlled processes satisfying \eqref{integral-constraint} will be called an {\it optimal process} for $F(\mathbf{x},u)$ under \eqref{state1} and \eqref{integral-constraint}. 
Following the Euler's terminology, integral constraints of the type \eqref{integral-constraint} are often named {\it isoperimetric constraints} and we will follow this convention, motivated by the fact that in the next section we will rephrase a geometrical isoperimetric problem as an optimal control problem and  \eqref{integral-constraint} will play exactly the role of the constraint on the perimeter. As it is customary, we will use the self-explanatory notations $\nabla_{\mathbf{x}}f$, $\nabla_{\mathbf{x}}g$, $\nabla_{\mathbf{x}}\mathbf{h}$, $\partial_{u} f$, $\partial_{u} g$, $\partial_{u} \mathbf{h}$ and so on, to denote the partial derivatives of $f$, $g$, $\mathbf{h}$.

\begin{theorem}[Pontryagin Principle]\label{PMP}
Let $(\mathbf{x},u)\in W^{1,1}_{per}(I,\mathbb{R}^{2})\times L^{\infty}(I,J)$ be an optimal process $F(\mathbf{x},u)$ under \eqref{state1} and \eqref{integral-constraint}. Then there exist $\lambda \geq 0$, $\mu \in \mathbb{R}$ and $\mathbf{p}\in W^{1,1}(I,\mathbb{R}^{2})$ not all of them trivial such that, for almost all $t\in I$,
\begin{equation}
\label{adjoint}
\dot{\mathbf{p}}(t)= \mathbf{p} \cdot \nabla_{\mathbf{x}} \mathbf{h}\big(\mathbf{x}(t)),u(t)\big)+ \mu \,  \nabla_{\mathbf{x}} g \big(\mathbf{x}(t)),u(t)\big) - \lambda \, \nabla_{\mathbf{x}} f \big(\mathbf{x}(t)),u(t)\big)
\end{equation}
and the optimal control $u$ satisfies, for all $t\in I$, the optimality condition
\begin{equation}
\label{maximality}
\begin{split}\mathbf{p}(t) \cdot \mathbf{h} & \big(\mathbf{x}(t)),u(t)\big)+ \mu  g \big(\mathbf{x}(t)),u(t)\big) - \lambda  f \big(\mathbf{x}(t)),u(t)\big)  \\ &  =\max_{v \in J }  \left\{ \mathbf{p}(t) \cdot \mathbf{h}\big(\mathbf{x}(t)),v \big)+ \mu  g \big(\mathbf{x}(t)),v \big) - \lambda  f \big(\mathbf{x}(t)),v\big) \right\}.
\end{split}
\end{equation}
\end{theorem}
The differential system \eqref{adjoint} takes the name of {\it adjoint system} and it is nothing but the Euler-Lagrange equation derived as a stationarity condition on the Lagrangian of the optimal problem (cf. \cite[\S 4.2.2]{MR924574}). 

We will use a couple of consequences of Theorem \ref{PMP} when applied to costrained problems arising in plane convex geometry. To this aim in the following corollary  we specify the Pontryagin's conditions for optimality in one dimensional control problems with a second order differential constraint.
\begin{corollary}
\label{Cor-PMP-01}
Given $f\in C^{1}(\mathbb{R}^{2})$ and $g \in C^{1}(\mathbb{R}^{2})$ and a constraint $C_{0}$. Let the pair $(x,u)\in W^{2,1}_{per}(I, \mathbb{R})\times L^{\infty}(I,J)$ be a minimizer of the functional 
\[
F(x,u) := \int_{a}^{b} f\big(x(t),u(t) \big) dt  
\]
among all the admissible pairs satisfying the differential constraint
\begin{equation}\label{state-02}
x(t)+\ddot{x}(t)=u(t)
\end{equation}
and the integral constraint
\begin{equation}\label{integral-constraint-02}
\int_{a}^{b} g\big(x(t),u(t)\big)dt  = C_{0}.
\end{equation}
Then there exist $\lambda \geq 0$, $\mu \in \mathbb{R}$ and $p \in W^{1,1}(I,\mathbb{R})$ not all of them trivial such that, for almost all $t\in I$,  $p$ is a solution of the equation
\begin{equation}
\label{adjoint-02}
\ddot{p}(t)+p(t)= + \mu \,  \partial_{x} g \big(x(t),u(t)\big) - \lambda \, \partial_{x} f \big(x(t),u(t)\big)
\end{equation}
and the optimal control $u$ satisfies, for all $t\in I$, the optimality condition
\begin{equation}
\label{maximality-02}
 \begin{split} p(t) u(t) &  + \mu g\big(x(t),u(t)\big) - \lambda f\big( x(t),u(t)\big) \\ 
 	& = \max_{v \in J } \left\{ p(t) v + \mu g\big(x(t),v\big) - \lambda f\big( x(t),v\big) \right\}.
\end{split}
\end{equation}
\end{corollary}
\begin{proof}
The proof easily follows by rewriting the differential constraint as a system of first order equation and applying Theorem \ref{PMP} with the phase variable $\mathbf{x}(t)=(x_{1}(t),x_{2}(t))$ that satisfies the system 
\[
\left\{
\begin{array}{l}
\dot{x_{1}}= x_{2} \\
\dot{x_{2}}=u-x_{1}
\end{array}
\right. . 
\]
The functionals involved are independent of the auxiliary variable $x_{2}$ and Theorem \ref{PMP} provides the existence of the multipliers exist $\lambda$, $\mu$ and $\mathbf{p}:=(p_{1},p_{2})$ satisfying the adjoint system 
\[
\left\{
\begin{array}{l}
\dot{p_{1}}= p_{2}+ \mu \,  \partial_{x} g \big(x(t)),u(t)\big) - \lambda \, \partial_{x} f \big(x(t)),u(t)\big) \\
\dot{p_{2}}=-p_{1}
\end{array}
\right.
\]
and the maximality condition
\[
 \begin{split} p_{1}(t)x_{2}(t) & +p_{2}(t) \big(u(t)-x_{1}(t) \big)  + \mu g\big(x_{1}(t),u(t)\big) - \lambda f\big( x_{1}(t),u(t)\big) \\ 
 	& = \max_{v \in J } \left\{ p_{1}(t)x_{2}(t)+p_{2}(t) \big(v-x_{1}(t) \big) + \mu g\big(x_{1}(t),v\big) - \lambda f\big( x_{1}(t),v\big) \right\}.  \\
	& = p_{1}(t)x_{2}(t) - p_{2}(t) x_{1}(t) + \max_{v \in J } \left\{ p_{2}(t) v + \mu g\big(x_{1}(t),v\big) - \lambda f\big( x_{1}(t),v\big) \right\}.
\end{split}	
\]
These equations are easily seen to be equivalent to \eqref{adjoint-02} and \eqref{maximality-02} setting $p=p_{2}$ and $x=x_{1}$. 

\end{proof}

In the special case when the functional $F$ and the isoperimetric constraint are linear in the control variable $u$, we can further deduce a {\it bang-bang} type condition for optimal controls. The following corollary easily follows from the previous one from the optimality condition \eqref{maximality-02} (being linear in the $v$ variable). 

\begin{corollary}
\label{Cor-PMP-02}
Under the same assumptions of Corollary \ref{Cor-PMP-01}, if we further assume that $f(x,u)=a(x)\,u$ and $g(x,u)=b(x)\,u$ with $a,b\in C^{1}(I)$, then we have
\begin{equation}\label{bang-01}
u(t)= \begin{cases}
\beta & \text{if } \;\; p(t)+\mu\, b(x(t)) - \lambda \,a(x(t)) > 0 \\
\alpha & \text{if } \;\; p(t)+\mu\, b(x(t)) - \lambda \, a(x(t)) < 0 
\end{cases}
\end{equation}
where  $\alpha:= \min\{t \,: \,t \in J\}$ and $\beta := \max\{t\,:\, t\in J \}$. \end{corollary}
\begin{remark}
Let us remark that if the set
\[
S:=\{ t\in I \; :\; p(t)+\mu\, b(x(t)) - \lambda \,a(x(t)) = 0\}
\]
has zero Lebesgue measure, then $u$ is almost everywhere determined by \eqref{bang-01}. This is the case for instance if $(x,u)$ is an optimal control with a a non-singular trajectory (cf. \cite{MR2201076}).
\end{remark}

\section{Main result}
\label{SEC:Main}
The main results of the paper are stated in the following theorem.

\begin{theorem}\label{thm-main-inequality}
Let $0\leq \alpha<\beta<\infty$. For any $K\in \mathcal{K}^{2}$, planar $(\alpha,\beta)$-convex body such the $2\pi \alpha  < P(K) < 2\pi \beta$,  the following inequality holds true:
\begin{equation}
\label{rev-isop}
A (K)\geq \frac 12 (\beta+\alpha)(P(K)-2\pi \alpha)+\pi \alpha^2- (\beta-\alpha)^2 \sin\left(\frac{P(K)-2\pi \alpha}{ 2(\beta-\alpha)}\right).
\end{equation}
Moreover the equality holds if and only if $K$ is the $(\alpha,\beta)-$egg.
\end{theorem}

\begin{remark}
\label{thm-Main-geometric}
Let $0\leq\alpha<\beta<\infty$ and $L\in (2\pi \alpha ,2\pi \beta)$. Then, modulo proper rigid transformations, the $(\alpha,\beta)$-egg is the unique minimizer of the area functional among all the $(\alpha,\beta)$-convex bodies in the plane with given perimeter equal to $L$. 
\end{remark}

From now on in this section we will implicitly assume that $\alpha, \beta$ and $L$ are fixed in such a way that $0 \leq \alpha<\beta<\infty$ and $L\in (2\pi \alpha ,2\pi \beta)$. The proof of the main theorem will be a consequence of the following lemmata. In the next one we prove the existence of an optimal set. Its convexity is ensured by Lemma \ref{lemma-continuity}.

\begin{lemma}
\label{Lem-existence}
  The shape optimisation problem
  \begin{equation}\label{P}
  \min\big\{A(K): K\in \mathcal{K}^{2}  \mbox{ is an } (\alpha,\beta)-\mbox{convex body with } P(K)=L\big\}
  \end{equation}
  admits at least a solution.
\end{lemma}
\begin{proof}
The proof follows by the direct methods of Calculus of Variations. 
Any minimizing sequence $K_n$ is bounded. Indeed, all the competitors are convex sets with perimeter and area equi-bounded, therefore also their diameters are equi-bounded (see for example \cite[Lemma 4.1]{Esposito2005}). 
By Blaschke selection theorem, up to extracting a subsequence, $K_n$ converges to a convex body $K_{\infty}$ in the Hausdorff metric. Lemma \ref{lemma-continuity} ensures that  $K_{\infty}$ is an admissible set and the conclusion follows by continuity of the perimeter and area functionals on $(\mathcal{K}^{2},d_{H})$. 
\end{proof}
In the next lemma we derive the optimality conditions for our minimization problem that is, the ODE (\ref{eq-optimality-for-K}). The argument is similar to that one used in \cite{Borisenko_2014}. 
We observe that a delicate point is to prove that the set of points where the support function $p_K$ is equal to $\gamma$ is finite.
\begin{lemma}
\label{Lem-optimality-condition}
Let $K\in\mathcal{K}^{2}$ be a minimizer for problem \eqref{P}. Then up to eventually translate $K$, there exists a constant $\gamma\in \mathbb{R}$ such that 
\begin{equation}\label{eq-optimality-for-K}
	\rho_{K}(t)=
	\left\{
  		\begin{array}{ll}
    		\beta, &\mbox{for} \quad p_{K}(t) < \gamma  \\
    		\alpha,        &\mbox{for} \quad  p_{K}(t) >\gamma
    	\end{array}
    \right., \;\; \text{a.e.}\;\; t \in (0, 2\pi),
 \end{equation}
 where with $p_{K}$, with a slight abuse of notation, we denoted the parametrised support function of the eventual translation of $K$ and $\rho_{K}(t)$ is the radius of curvature of $\partial K$. Moreover the set $S:=\{ t \in [0, 2\pi)\;:\; p_{K}(t) = \gamma\}$ is finite.
\end{lemma}
\begin{proof}

We start observing that by Propositions \ref{Properties-01} we can identify a given admissible set $K\in \mathcal{K}^{2}$ with its parametric support function $p_{K}$ and by  Proposition \ref{Properties-02} we can rephrase the minimization problem \eqref{P} as an optimal control problem. 
If $K\in\mathcal{K}^{2}$ is a minimizer for problem \eqref{P}, then the pair given by its support function and its parametric radius of curvature, i.e.  $(x,u)=(p_{K},\rho_{K}) \in W^{2,1}_{per}(I, \mathbb{R})\times L^{\infty}(I,J)$ with $I=(0,2\pi)$ and $J=[\alpha, \beta]$, form indeed an optimal control process for the functional 
\[
F(x,u) :=\frac{1}{2}\int_{0}^{2\pi} u(t) \, x(t)\, dt 
\]
under the differential constraint
\[ 
x(t)+\ddot{x}(t)=u(t)
\]
and the isoperimetric one
\[
 \int_{0}^{2\pi}u(t)\, dt=L.
\]
We can therefore use Corollary \ref{Cor-PMP-02} to deduce that there exist $\lambda \geq 0$, $\mu \in \mathbb{R}$ and $s \in W^{1,1}(I,\mathbb{R})$ not all of them trivial such that
\begin{equation}\label{bang-for-Rho}
\rho_{K}(t)= \begin{cases}
\beta & \text{if } \;\; s(t)+\mu - \frac{\lambda}{2} \,p_{K}(t) > 0 \\
\alpha & \text{if } \;\; s(t)+\mu - \frac{\lambda}{2} \, p_{K}(t) < 0 
\end{cases}
\end{equation}

Moreover, from Corollary \ref{Cor-PMP-01}, the multiplier $s(t)$ solves the adjoint equation
\[
\ddot{s}(t)+s(t)=-\frac{\lambda}{2} \,\rho_{K}(t),
\]
that, together with the differential constraint written for the pair $(p_{K},\rho_{K})$, implies that the function $\beta:=s+\frac{\lambda}{2} p_{K}$ is a $2\pi$-periodic solution of the ordinary differential equation $y(t)+\ddot{y}(t)=0$. Therefore there exist constants $c_{1}$ and $c_{2}$, such that
\begin{equation}\label{EQ:adjoint-S}
s(t)+\frac{\lambda}{2} p_{K}(t)= c_1\cos(t)+c_2 \sin(t).
\end{equation}
We can therefore rewrite \eqref{bang-for-Rho} ,  as
\[
\rho_{K}(t)= \begin{cases}
\beta & \text{if } \;\; c_1\cos(t)+c_2 \sin(t) + \mu - \lambda \,p_{K}(t) > 0 \\
\alpha & \text{if } \;\;c_1\cos(t)+c_2 \sin(t) + \mu - \lambda \,p_{K}(t)< 0 
\end{cases}
\]

We claim that $\lambda \not=0$. If not, first we observe that by non-triviality condition of the Pontryagin principle, $\mu$ and $s(t)$ cannot be simultaneously identically zero. If $s=0$, then from \eqref{bang-for-Rho} we have 
\[
\rho_{K}(t)= \begin{cases}
\beta & \text{if } \;\; \mu  > 0 \\
\alpha & \text{if } \;\; \mu < 0 
\end{cases}.
\]
Being $\mu\not=0$, $K$ is a circle of radius $\alpha$ or $\beta$ that is not an admissible set. If instead $s \not=0$, the adjoint equation 	\eqref{EQ:adjoint-S} ensures us that $s(t)=c_1\cos(t)+c_2 \sin(t)= A\cos(t+\phi)$, with $A \not= 0$ and $\phi$ constant. The condition \eqref{bang-for-Rho} becomes
\[
\rho_{K}(t)= \begin{cases}
\beta & \text{if } \;\; A\cos(t+\phi) + \mu  > 0 \\
\alpha & \text{if } \;\; A\cos(t+\phi) - \mu < 0 
\end{cases}.
\]
Therefore, since the equation $A\cos(t+\phi) + \mu =0$ admits at most two solutions in the interval $[0,2\pi)$, it follows that $\partial K$ is the union of at most two arcs of circle with radii $\alpha$ and $\beta$. This is impossible for an $(\alpha,\beta)$-convex set by the regularity Lemma \ref{Regularity-B} (cf. Remark \ref{REM-consecArc}).  This proves the claim. 
 
Since  $\lambda>0$, in a translated coordinate system centered in $\left(\frac{c_{1}}{\lambda},\frac{c_{2}}{\lambda}\right)$, the parametric support function of $K$ will change in $p_{K}(t) -\frac{c_{1}}{\lambda} \cos(t) - \frac{c_{2}}{\lambda}\sin(t)$. Therefore \eqref{bang-for-Rho} will read
\[
\rho_{K}(t)= \begin{cases}
\beta & \text{if } \;\; \,p_{K}(t) <  \frac{\mu}{\lambda}=:\gamma\\
\alpha & \text{if } \;\;p_{K}(t) >  \frac{\mu}{\lambda} =:\gamma
\end{cases}
\]

We now prove that $S:=\{ t \in [0, 2\pi]\;:\; p_{K}(t) = \gamma\}$ is a finite set. Let $t_{0}\in [0,2\pi] \in S^{c}$, say $p_{K}(t_{0}) > \gamma$, and let $(a_{0},b_{0})$ be  the connected component of the set $S^{c}$, containing $t_{0}$ (more explicitly we can define $a_{0}= \inf\{ \tilde t \;:\; p_{K}(t) >\gamma \;\forall \, t\in (\tilde t, t_{0}) \}$ and $b_{0}= \sup\{ \tilde t \;:\; p_{K}(t) >\gamma \;\forall \, t\in (t_{0}, \tilde t) \}$). Observe that by continuity of $p_{K}$ we deduce that
 \begin{equation}\label{continuity_after_AH}
 p_{K}(a_{0})=p_{K}(b_{0})=\gamma.
 \end{equation}
Moreover we can uniquely solve the equation  $p_{K}+p_{K}''= \alpha$ in $(a_{0},b_{0})$, and therefore deduce the existence of  two constants $C_{1}$ and $C_{2}$ such that $p_{K}(t)=C_1\cos t+C_2\sin t +\alpha$ in $[a_{0},b_{0}]$.
We claim that $b_{0}$ is an isolated point for $S$. The same argument could be applied for the left endpoint $a_{0}$. By contradiction, let $\{t_{m}\}_{m\in \mathbb{N}}$ with $t_{m}> b_{0}$,  $p_{K}(t_{m})= \gamma$ and such that $t_{m} \to b_{0}^{+}$. The regularity of $p_{K}$ ensures that the left and right derivatives of $p_{K}$ in $b_{0}$ agree. We can write (recalling the explicit expression of $p_{K}$ in $(a_{0},b_{0})$)
\[
-C_{1}\sin b + C_{2}\cos b = p_{K}'(b-)=p_{K}'(b+)=\lim_{m\to \infty}\frac{p_{K}(t_m)-p_{K}(b)}{t_m-b}=0. 
\]
The last equality, together with \eqref{continuity_after_AH}, tells us that the couple $(C_{1},C_{2})$ solves the following linear system
\[
\begin{cases}
C_1\cos a_{0}+C_2\sin a_{0}  = \gamma -\alpha   \\ 
C_1\cos b_{0}+C_2\sin b_{0}  = \gamma -\alpha   \\
-C_{1}\sin b_{0} + C_{2}\cos b_{0} = 0 
\end{cases},
\]
that is solvable only if (imposing the determinant of the full matrix to be zero) 
\[
(\gamma- \alpha)(\cos(a_{0}-b_{0})-1)=0
\]
that in turns implies $\gamma=\alpha$ or $b_{0}-a_{0}=2\pi$. The last equality means that $x$ represent a full circle of radius $\alpha$, that (if we choose $L > 2\pi \alpha$) is not an admissible competitor for our problem.
It remains to study the case $\gamma=\alpha$, that leads easily to a contradiction by observing that the only solution od the linear system is $(C_{1},C_{2})=(0,0)$ and therefore $x(t)=\alpha=\gamma$ for any $t \in (a_{0},b_{0})$ against the definition of $S^{c}$ and this proves the claim. An analogous argument can be done when $p_{K}(t_{0})<\gamma$, with $\beta$ in place of $\alpha$. Since the connected components of $S^{c}$ have isolated endpoints, they are finite in number. Finally we have proved that $S^{c}$ is a finite union of disjoint  relatively open intervals in $[0,2\pi]$. Therefore its complement $S$ is a finite union of, possibly degenerate, relatively closed intervals in $[0,2\pi]$. With the same argument as above, it is easy to prove that the interior of $S$ is empty. Indeed it is sufficient to argue by contradiction and use the regularity of $p_{K}$ at the endpoints of the connected components of $S$ with non empty interior. 
\end{proof}

From now on, our technique is completely different from that one of \cite{Borisenko_2014}. Indeed we exploit equation (\ref{eq-optimality-for-K}) of Lemma
\ref{Lem-optimality-condition} in an analytical way.
In the next lemma we prove that for an optimal set, the arcs of radii $\alpha$ are congruent to each other, as well as 
the arcs of radii $\beta$.

\begin{lemma}
\label{Lem-N-gone}
Any $(\alpha,\beta)$-convex body $K\in \mathcal{K}^{2}$ that satisfies \eqref{eq-optimality-for-K} is necessarily an $(\alpha,\beta)$-regular $N$-gone.
\end{lemma}
\begin{proof}
From Lemma \ref{Lem-optimality-condition}, we infer that $\partial K$ is the union of a finite number of arcs of circles with radii $\alpha$ and $\beta$, being the radius of curvature, $\rho_{K}$, a piecewise constant function with a finite number of jumps, assuming only two values. Moreover in any jump point $t\in [0,2\pi)$ of $\rho_{K}$, it holds $P_{K}(t)=\gamma$. By Remark \ref{REM-consecArc} we can easily deduce that the arcs are even in number and the radii alternate between the values $\alpha$  and $\beta$. In the case $\alpha=0$, we will have arcs of radius $\beta$, joining each other at corner points.

We can therefore assume that $\partial K$ is made of  $2N$ disjoint arcs. The parametric support function of $K$ can be written as
\[
p_{K}(t)=
 \left\{
 \begin{array}{ll}
   C_1^{2i+1} \cos t + C_2^{2i+1} \sin t + \beta \,,& t\in [t_{2i+1}, t_{2i+2}] \\ \\
   C_1^{2i+2} \cos t + C_2^{2i+2} \sin t + \alpha \,,& t\in [t_{2i+2}, t_{2i+3}]
 \end{array}
 \right.,  \;\; i\in{0,1,\dots,N-1}.
\]
with $ \{t_1<t_2<t_3 \cdots <t_{2N+1}=t_1+2\pi\}$ and $(C_1^{2i+1},C_2^{2i+1})$ are the coordinates of centers of the disks of radius $\beta$ and $(C_1^{2i+2},C_2^{2i+2})$ are the coordinates of centers of the disks of radius $\alpha$. Imposing the continuity in $t_{i}$ one gets
\begin{equation}\label{CIJ}
C_1^{2i+1}= (\gamma-\beta) \frac{\sin t_{2i+2}  - \sin t_{2i+1} }{ \sin (t_{2i+2}-t_{2i+1})   }
;\quad
C_2^{2i+1}= (\gamma-\beta)\frac{\cos t_{2i+2} -\cos t_{2i+1} }{\sin (t_{2i+1}-t_{2i+2}) }.
\end{equation}
   For $0\leq i \leq N-1$, one gets
\begin{equation}\label{CIJ'}
C_1^{2i+2}= (\gamma-\alpha) \frac{\sin t_{2i+3}  - \sin t_{2i+2} }{ \sin (t_{2i+3}-t_{2i+2})   }
;\quad
C_2^{2i+2}= (\gamma-\alpha)\frac{\cos t_{2i+3} -\cos t_{2i+2} }{\sin (t_{2i+2}-t_{2i+3}) }.
\end{equation}
From \eqref{CIJ} and \eqref{CIJ'} we easily deduce that $\gamma \not= \alpha $ and  $\gamma \not= \beta$, otherwise the arcs of the circles of radius $\alpha$ or $\beta$ contained in $\partial K$ should lie all on the same circle centered at the origin.

The continuity of the derivative of the parametric support function in $t_{j}$ ensures us that, for $0\le j \le 2N-1$,
\begin{equation}\label{systeme1}
  C_{1}^{j}\sin(t_{j+1})-C_{2}^{j}\cos(t_{j+1}) =C_{1}^{j+1}\sin(t_{j+1})-C_{2}^{j+1}\cos(t_{j+1}) 
\end{equation}

Combining the relations \eqref{CIJ}, \eqref{CIJ'} and \eqref{systeme1}, we can write, for $0\le i \le N-1$, 
\[
\left\{
\begin{array}{lllr}
\displaystyle (\gamma-\beta) \frac{1-\cos (t_{2i+2} -t_{2i+1}) }{ \sin (t_{2i+2} -t_{2i+1}) }&=&\displaystyle  (\gamma-\alpha)\frac{-1+\cos (t_{2i+3} -t_{2i+2}) }{ \sin (t_{2i+3} -t_{2i+2})} &
\\ \\
\displaystyle(\gamma-\alpha)\frac{1-\cos (t_{2i+3} -t_{2i+2}) }{ \sin (t_{2i+3} -t_{2i+2})   }  &=& \displaystyle  (\gamma-\beta) \frac{-1+\cos (t_{2i+4} -t_{2i+3}) }{ \sin (t_{2i+4} -t_{2i+3})}& \end{array}
\right. .
\]
Since the function $t \to \frac{1-\cos (t) }{ \sin (t) }$ is strictly monotone, from the previous system we infer the existence of two positive constants  $\tau$ and $\sigma$ such that $t_{2i+2}-t_{2i+1}=\tau$ and $t_{2i+3}-t_{2i+2}=\sigma$ for any $0\le i \le N-1$. This proves the claim.
\end{proof}

\begin{remark}
\label{REM:ValueLambda}
Let $K\in \mathcal{K}^{2}$ be an $(\alpha,\beta)$-regular $N$-gone with perimeter $P(K)=L$, as in the Example \ref{N-gone}. Suppose that up to a translation of $K$, the values of the parametric support functions in the points $t_{j}$ are constants, i.e. there exists $\lambda$, such that $\lambda_{j}=\lambda$ for any $j\in \{0,1,\dots,N-1\}$.  As a byproduct of the proof of the previous lemma, we can explicitly calculate the value of $\lambda=p_{K}(t_{j})$. Indeed, simply by imposing the continuity of the parametric support function in $t_{i}$ and solving the linear system, we get
\[
\lambda= \frac{\beta[1-\cos(\tau)]\sin(\sigma)+\alpha[1-\cos(\sigma)]\sin(\tau)}{[1-\cos(\tau)]\sin(\sigma)+[1-\cos(\sigma)]\sin(\tau)}.
\]
\end{remark}

\begin{lemma}
\label{Lem-best-N-gone}
Let $K\in \mathcal{K}^{2}$ be an $(\alpha,\beta)$-regular $N$-gone with $P(K)=L$ that satisfies \eqref{eq-optimality-for-K}
then
\begin{equation}
\label{eq-area-n-gone}
A(K)=\frac{\beta+\alpha}{2}(L-2\pi \alpha)+\pi \alpha^2
 +(\beta-\alpha)^2\frac{N\big(\cos(\frac{\pi}{N})-\cos(\frac{\pi(\beta+\alpha)-L}{N(\beta-\alpha)})\big)}{2\sin(\frac{\pi}{N})}.
\end{equation}
Moreover the minimum value for the area functional is realized for $N=2$, i.e. for the $(\alpha,\beta)$-egg.
\end{lemma}

\begin{proof}
Let $K\in \mathcal{K}^{2}$ be an $(\alpha,\beta)$-regular $N$-gone and using the same notation as in Example \ref{N-gone} for its parametric support function and radius of curvature, we can compute
\[
\begin{split}
 2A(K) = &  \sum_{i=1}^{N}\int_{t_{2i-1}}^{t_{2i}} \beta \big[C_1^{2i-1}\cos(t)+ C_2^{2i-1}\sin(t)+\beta\big]dt \\
       &+ \sum_{i=1}^{N}\int_{t_{2i}}^{t_{2i+1}} \alpha \big[C_1^{2i}\cos(t)+ C_2^{2i}\sin(t)+\alpha\big]dt \\
    = & \sum_{i=1}^{N} \bigg\{\beta \left[C_1^{2i-1}(\sin(t_{2i})-\sin(t_{2i-1}))+ C_2^{2i-1}(\cos(t_{2i-1})-\cos(t_{2i}))\right]  \\
       & +  \alpha \left[C_1^{2i}(\sin(t_{2i+1})-\sin(t_{2i}))+C_2^{2i}(\cos(t_{2i})-\cos(t_{2i+1}))\right]\bigg\} \\
       & +\beta^2(t_{2i}-t_{2i-1})+\alpha^2(t_{2i+1}-t_{2i}) .
       \end{split}
\]

Therefore, using \eqref{CIJ} and \eqref{CIJ'} in the previous formula and  recalling that $t_{2i}  - t_{2i+1}=\tau$ and $t_{2i+1}  - t_{2i}= \sigma$, we infer
\[
A(K) = \beta(\lambda-\beta)N \; \frac{\big[1-\cos(\tau)\big]}{\sin(\tau)} + \alpha(\lambda-\alpha) N \; \frac{\big[1-\cos(\sigma)\big]}{\sin(\sigma)}+ \frac{N}{2}(\beta^2 \tau +\alpha^2\sigma ).
\]
Finally, using the value of $\lambda$ given by Remark \ref{REM:ValueLambda}, we get 
\[
\lambda-\beta = (\alpha-\beta)\frac{1-\cos(\sigma)}{
\sin(\sigma)+\sin(\tau)-\sin(\sigma+\tau)} \sin(\tau)
\]
and
\[
\lambda-\alpha = (\beta-\alpha)\frac{1-\cos(\tau)}{
\sin(\sigma)+\sin(\tau)-\sin(\sigma+\tau)} \sin(\sigma).
\]
Therefore we can write 
\begin{equation}\label{cinquo}
A(K ) =  \frac{N}{2}(\beta^2 \tau +\alpha^2\sigma )-\frac{N(\beta-\alpha)^2\big[1-\cos(\sigma)\big]\big[1-\cos(\tau)]}{
\sin(\sigma)+\sin(\tau)-\sin(\sigma+\tau)}.
\end{equation}
Using the elementary relations
\[
\sin (a) + \sin (b) +\sin (c) - \sin (a+b+c) = 4 \sin \left(\frac{a+b}{2} \right)\sin \left(\frac{b+c}{2} \right) \sin \left(\frac{a+c}{ 2} \right)
\]
and
\[
\big(1-\cos(a)\big)\big(1-\cos(b)\big)=4\sin^2(a/2)\sin^2(b/2)
\]
from \eqref{cinquo}, we get
\[
A(K)= \frac{N}{2}(\beta^2 \tau +\alpha^2\sigma )-N(\beta-\alpha)^2 \frac{\sin \left(\frac{\sigma}{2}\right) \sin\left(\frac{\tau}{2}\right)}
{\sin\left(\frac{\sigma+\tau}{2}\right)}.
\]
To make explicit the dependence on $N$ in the expression of the area we introduce the auxiliary variable (cf. \eqref{EQ:sigma-tau}) 
\[ 
\omega:= N \frac{\sigma}{2} = \frac{L-2\pi\alpha}{2(\beta-\alpha)}.
\] 
Recalling that $N(\sigma+\tau)=2\pi$ we can finally write the area of the $(\alpha,\beta)$-regular $N$-gone $K$ as
\[
A(K)= \beta^{2} \pi - (\beta^2  +\alpha^2) \frac{\omega}{2} - (\beta- \alpha)^{2} \Phi (N,\omega),
\]
where we have set
\[
\Phi(N,\omega):= N \frac{\sin \left(\frac{\omega}{N}\right) \sin\left(\frac{\pi}{N}-\frac{\omega}{N}\right)}{\sin\left(\frac{\pi}{N}\right)}
\] 
And this proves \eqref{eq-area-n-gone}. 

To prove that the minimum value of the area is attained when $N=2$,  which corresponds to the area of the $(\alpha,\beta)-$egg, it is sufficient to prove that 
\[
\Phi(N,\omega) \leq \Phi(2,\omega) = \sin(\omega).
\]
To this aim, we observe that $\omega \in [0,\pi]$ and we show that the function
\[
f_N(x):=\sin (x) \sin\left(\frac{\pi}{N}\right)-N\sin\left(\frac{\pi-x}{N}\right)\sin\left(\frac{x}{N}\right)
\]
is positive for $x\in [0,\pi]$ and $N\geq 2$. We observe that $f_N(0)=f_N(\pi)=0$ and that $f_N(x)$ is symmetric with respect to $x_{s}=\frac{\pi}{2}$.
We claim that $f_N$ is  increasing in $[0,\frac{\pi}{2}]$. This will imply that $f_N$ is positive on $[0,\pi]$.
For that, we first observe that the function 
\[
h_N(x): =\frac{\sin(\frac{\pi-2x}{N})}{\sin(\frac{\pi-2x}{N+1})}
\] 
is  increasing  on $(0,\frac{\pi}{2})$. Indeed its derivative 
\[
h_N'(x)=-2\frac{\sin(\frac{\pi-2x}{N})}{N(N+1)\sin(\frac{\pi-2x}{N+1})}\left[(N+1)\cot(\frac{\pi-2x}{N})-N\cot(\frac{\pi-2x}{N+1})\right]\,,
\]
satisfies $h_N'(x)>0$ on $(0,\frac{\pi}{2})$ since $x\mapsto x\cot x$ is decreasing on $\left(0,\frac{\pi}{2}\right)$ and therefore 
\[
(N+1)\cot\left(\frac{\pi-2x}{N}\right)-N\cot\left(\frac{\pi-2x}{N+1}\right)<0.
\]
The monotonicity of $h_N$, implies that for $N\geq 2$ 
\[
\frac{\sin(\frac{\pi}{N+1})}{\sin(\frac{\pi-2x}{N+1})}\geq \frac{\sin(\frac{\pi}{N})}{\sin(\frac{\pi-2x}{N})}\,, 
\]
that is, for any $x\in \left(0, \frac{\pi}{2}\right)$, the sequence $\{a_{N}\}_{N=2}^{\infty}$
\[
a_{N}:= \frac{\sin(\frac{\pi}{N})}{\sin(\frac{\pi-2x}{N})}
\]
is monotone in $N$, and therefore, for $N\geq 2$, it holds $a_{N}\geq a_{2}$, that reads as 
\[
\frac{\sin(\frac{\pi}{N})}{\sin(\frac{\pi-2x}{N})}\geq \frac{\sin(\frac{\pi}{2})}{\sin(\frac{\pi-2x}{2})} = \frac{1}{\cos(x)}.
\]
The last inequality is equivalent to say that
\[
f_N'(x)=\cos x\sin(\frac{\pi}{N})-\sin(\frac{\pi-2x}{N})>0 \;,\;\; x \in  \left(0, \frac{\pi}{2}\right),
\]
proving the claim and the lemma.
\end{proof}

\section*{Acknowledgements}
The first and the second author were partially supported by the project ANR-18-CE40-0013 SHAPO financed
by the French Agence Nationale de la Recherche (ANR). The research of G.P. was partially supported by project Vain-Hopes within the program VALERE: VAnviteLli pEr la RicErca and by the INdAM-GNAMPA group. 

The authors are grateful to Antoine Henrot for some very useful discussions.

\bibliographystyle{plain}

\end{document}